\documentclass[11pt]{amsart}

\usepackage[a4paper,hmargin=3.5cm,vmargin=4cm]{geometry}
\usepackage{amsfonts,amssymb,amscd,amstext,verbatim}
\usepackage{graphicx}
\usepackage[dvips]{epsfig}

\usepackage[utf8]{inputenc}
\usepackage{hyperref}

\usepackage{verbatim}
\usepackage{fancyhdr}
\input xy
\xyoption{all}

\usepackage{times}

\usepackage{enumerate}
\usepackage{titlesec}
\usepackage{mathrsfs}

\titleformat{\section}
{\filcenter\bfseries\large} {\thesection{.}}{0.2cm}{}
\titleformat{\subsection}[runin]
{\bfseries} {\thesubsection{.}}{0.15cm}{}[.]
\titleformat{\subsubsection}[runin]
{\em}{\thesubsubsection{.}}{0.15cm}{}[.]

\usepackage[up,bf]{caption}

\newtheorem{theorem}{Theorem}[section]

\newtheorem{lemma}[theorem]{Lemma}

\theoremstyle{definition}
\newtheorem{definition}[theorem]{Definition}
\newtheorem{remark}[theorem]{Remark}

\numberwithin{equation}{section}
\numberwithin{figure}{section}

\newcommand\Ocal{\mathcal{O}}

\newcommand\C{\mathbb{C}}
\newcommand\D{\overline{\mathbb D}}

\renewcommand\D{\mathbb D}

\newcommand\R{\mathbb{R}}

\newcommand\igot{\mathfrak{i}}

\renewcommand\igot{\mathfrak{i}}

\renewcommand\imath{\igot}

\newcommand\hra{\hookrightarrow}

\newcommand\wt{\widetilde}

\newcommand\Aut{\mathrm{Aut}}

\usepackage{color}

\begin{document}

\fancyhead[LO]{Proper holomorphic immersions into Stein manifolds with the density property}
\fancyhead[RE]{F.\ Forstneri\v c} 
\fancyhead[RO,LE]{\thepage}

\thispagestyle{empty}

\vspace*{1cm}
\begin{center}
{\bf\LARGE Proper holomorphic immersions into Stein manifolds with the density property}

\vspace*{0.5cm}

{\large\bf  Franc Forstneri{\v c}} 
\end{center}


\vspace*{1cm}

\begin{quote}
{\small
\noindent {\bf Abstract}\hspace*{0.1cm}
In this paper we prove that every Stein manifold $S$ admits a proper holomorphic immersion into any Stein manifold 
of dimension $2\dim S$ enjoying the density property or the volume density property. 

\vspace*{0.2cm}

\noindent{\bf Keywords}\hspace*{0.1cm} Stein manifold, density property, proper holomorphic immersion

\vspace*{0.1cm}

\noindent{\bf MSC (2010):} \hspace*{0.1cm} 32E10, 32E20, 32E30, 32H02; 32Q99}


\end{quote}


\vspace{1.5mm}

%
%
%
%
%
%
%
%
%
\section{Introduction}
\label{sec:Intro}
A complex manifold $X$ is said to enjoy the {\em density property} 
if the Lie algebra generated by all the $\C$-complete holomorphic vector fields is dense in the 
Lie algebra of all holomorphic vector fields on $X$ (see Varolin \cite{Varolin2001,Varolin2000}
or \cite[Sect.\ 4.10]{Forstneric2017E}). Similarly one defines the {\em volume density property} of a 
complex manifold endowed with a holomorphic volume form, as well as the algebraic versions of these
properties (see Kaliman and Kutzschebauch \cite{KalimanKutzschebauch2008IM}). 
This important class has been the focus of intensive recent research for affine algebraic and Stein manifolds; see 
the papers \cite{Andrist2018JGA,AndristKutzschebauchLind2015,AndristKutzschebauchPoloni2017PAMS,
KalimanKutzschebauch2017MA,KalimanKutzschebauch2015,KalimanKutzschebauch2016TG,
KutzschebauchLeuenberger2016,KutzschebauchLeuenbergerLiendo2015,Leuenberger2016PAMS}, among others.
These manifolds are highly symmetric and enjoy the Anders\'en-Lempert property
\cite{AndersenLempert1992} on approximation of isotopies of injective holomorphic maps between Runge domains 
by holomorphic automorphisms; see \cite[Theorem 1.1]{ForstnericRosay1993} for $\C^n$ 
and \cite[Theorem 4.10.5, p.\ 143]{Forstneric2017E}. Furthermore, every complex manifold with the 
density property is an {\em Oka manifold}; see \cite[Proposition 5.6.23, p.\ 223]{Forstneric2017E}. 

The following general embedding theorem has been discovered recently.

%
%
\begin{theorem}
\label{th:AFRW}
{\rm (Andrist et al. \cite[Theorem 1.1]{AndristFRW2016}.)}
Let $X$ be a Stein manifold with the (volume) density property. 
If $S$ is a Stein manifold and $2\dim S +1\le \dim X$, then any continuous map $S\to X$ is 
homotopic to a proper holomorphic embedding $S\hra X$. 
\end{theorem}

In this paper we complete the picture by proving the corresponding result on the existence
of proper holomorphic immersions into manifolds of double dimension.

%
%
\begin{theorem}
\label{th:main}
Let $X$ be a Stein manifold enjoying the density property or the volume density property. 
If $S$ is a Stein manifold with $2\dim S=\dim X$, then any continuous map $S\to X$ is 
homotopic to a proper holomorphic immersion with simple double points.
\end{theorem}

An immersion $f\colon S\to X$ is said to have {\em simple double points}
if for any pair of distinct points $p,q\in S$ with $f(p)= f(q)$ the tangent 
planes $df_p(T_{p}S)$ and $df_q(T_{q}S)$ intersect trivially within $T_{f(p)}X$, and $f$ has no triple points.
Clearly, any such pair $(p,q)$ is isolated, and if $f$ is proper then the sequence $(p_j,q_j)\in S\times S$ 
of pairs of double points is such that each of the sequences $p_j$ and $q_j$ discrete in $S$.

The special case of Theorem \ref{th:main} with $X=\C^{2\dim S}$ is a classical 
theorem of Bishop \cite{Bishop1961} and Narasimhan \cite{Narasimhan1960AJM}.
When $S$ is an open Riemann surface and $\dim X=2$, Theorem \ref{th:main} was proved 
beforehand by Andrist and Wold \cite{AndristWold2014}. 

Let us mention a few related open problems.
Assume that $S$ and $X$ are Stein manifolds and $X$ has the (volume) density property.
Can one always find a proper holomorphic map $S\to X$ when $\dim S <\dim X$? 
This is possible for $X=\C^n$ according to the papers of Bishop and Narasimhan cited above. 
The second question is whether the immersion and the embedding dimensions 
in the above theorems can be lowered. Recall that every Stein manifold of dimension $d\ge 1$ 
admits a proper holomorphic immersion into the complex Euclidean space of dimension
$\left[\frac{3d+1}{2}\right]$, and if $d>1$ then it admits a proper holomorphic embedding into 
$\C^{\left[\frac{3d}{2}\right]+1}$ (see Eliashberg and Gromov \cite{EliashbergGromov1992} and 
Sch\"urmann \cite{Schurmann}). Examples of Forster \cite{Forster1970} show that this embedding
dimension is minimal for every $d>1$, and the immersion dimension is minimal for even values of $d$ 
and is off by at most one for odd values of $d$. The optimal non-proper immersion dimension is $\left[\frac{3d}{2}\right]$. 
Their proofs rely on the product structure of $\C^n$ and do not generalize to more general target manifolds.
The problem whether every open Riemann surface embeds properly holomorphically into $\C^2$
is still open, although considerable progress has been made in the last decade.
A discussion of these topics can be found in \cite[Chap.\ 9]{Forstneric2017E}.

%
%
The paper is organized as follows. In Sect.\ \ref{sec:preliminaries} we collect some preliminaries
and develop the notion of a {\em very special Cartan pair}
which plays an important role in the proof. In Sect.\ \ref{sec:mainlemma} we prove the main technical 
result, Lemma \ref{lem:main}. Although it is similar to \cite[Lemma 2.2]{AndristFRW2016},
its proof strongly relies on the fact that the attaching set of the convex bump in a very special Cartan pair
can be an arbitrarily thin convex slab. This allows us to ensure that the immersion into $X$, 
given in an induction step of the proof, is an embedding on the attaching set of the bump. 
This condition is  important in the rest of the argument where the (volume) 
density property of $X$ is used to approximate the given immersion 
on the attaching set by an immersion of the bump. The rest of the procedure,  
gluing these two maps etc., is the same as in \cite[proof of Lemma 2.2]{AndristFRW2016}.
With Lemma \ref{lem:main} in hand, Theorem \ref{th:main} is proved in exactly the same way as 
\cite[Theorem 1.1]{AndristFRW2016}.

\begin{remark}
By using the technique of this paper, it is possible to obtain a more precise version of Theorem \ref{th:main} in 
the spirit of \cite[Theorem 15]{ForstnericRitter2014} and \cite[Theorem 1.2]{AndristFRW2016},
i.e., with interpolation and the control of the image of the set $S\setminus K$ for a given 
compact $\Ocal(S)$-convex subset $K$ of the source Stein manifold $S$.
More precisely, assume that $S$ and $X$ are as in Theorem \ref{th:main},
$K\subset S$ is a compact $\Ocal(S)$-convex set, $U\subset S$ is a open set containing $K$, 
$S' \subset S$ is a closed complex subvariety of $S$, $L\subset X$ is  
a compact $\Ocal(X)$-convex set, and $f\colon U\cup S'\to X$ is a holomorphic map such that 
$f:S'\to X$ is an immersion with simple double points satisfying $f(bK \cup (S'\setminus K)) \cap L =\emptyset$. 
Then we can approximate $f$ uniformly on $K$ by a holomorphic immersion $F\colon S\to X$ 
with simple double points satisfying $F(S\setminus K) \subset X \setminus L$ and $F|_{S'}=f|_{S'}$. 
If in addition the map $f|_{S'} \colon S'\to X$ is proper (in particular, if $S'=\emptyset$),
then $F$ can also be chosen proper. 
\end{remark}

%
%
\begin{remark}
Lemma \ref{lem:main} also serves to complete the details in 
the proof of \cite[Theorem 15]{ForstnericRitter2014}, due to T.\ Ritter and the author,
in the second case when the compact set $L\subset \C^n$ in the statement of that theorem 
is polynomially convex and $n=2\dim X$.
(In \cite{ForstnericRitter2014}, $X$ denotes the source Stein manifold which 
corresponds to $S$ in the present paper, while the target manifold is $\C^n$.)
The proof of \cite[Theorem 15]{ForstnericRitter2014} is complete when $L$ is convex or 
holomorphically contractible, while the case when $L$ is polynomially convex 
and $n=2\dim X$ can be seen by supplementing the proof in \cite{ForstnericRitter2014}
by Lemma \ref{lem:main} in the present paper.
\end{remark}

%
%

\section{Preliminaries}
\label{sec:preliminaries}
Let  $\Ocal(S)$ denote the algebra of all holomorphic functions on a complex manifold $S$, 
endowed with the compact-open topology. A compact set $K$ in $S$ is said to be {\em $\Ocal(S)$-convex} 
if for every point $p\in S\setminus K$ there exists a function $g\in \Ocal(S)$ with $|g(p)| > \sup_K |g|$. 

If $S$ is a closed complex subvariety of a Stein manifold $X$, then a compact set
$K\subset S$ is $\Ocal(S)$-convex of and only if it is $\Ocal(X)$-convex.
We shall need a version of this result for immersed submanifolds with simple double points.

\begin{lemma} \label{lem:holoconvex}
Assume that $S$ and $X$ are Stein manifolds and $f\colon S\to X$ is a proper 
holomorphic immersion with only simple double points. 
Then a compact subset $K\subset S$ is $\Ocal(S)$-convex if and only if
its image $f(K)\subset X$ is $\Ocal(X)$-convex.
\end{lemma}

\begin{proof}
Let $\Delta_S$ denote the diagonal of $S\times S$.
The hypothesis on $f$ implies that there are at most a countably many pairs $(a_j,b_j)\in S\times S\setminus\Delta_S$ 
such that the sequences $a_j$ and $b_j$ are discrete in $S$, $f(a_j)=f(b_j)$ for all $j$, and any 
$(a,b)\in S\times S\setminus \Delta_S$ satisfying $f(a)=f(b)$ is one of the pairs $(a_j,b_j)$. 
The image $\Sigma=f(S)\subset X$ is a closed complex subvariety of $X$ whose only 
singularities are simple normal crossings at the points $c_j=f(a_j)=f(b_j)$.
Since $f$ is proper, the sequence $c_j\in \Sigma$ is discrete.

Assume that the compact set $K\subset S$ is $\Ocal(S)$-convex.
Let $p\in \Sigma\setminus f(K)$. If $p\ne c_j$ for all $j$,  then $p=f(q)$ for a unique point
$q\in S\setminus K$. Since the sequences $a_j,b_j$ are discrete in $S$,
the Cartan-Oka-Weil theorem gives a function $g\in \Ocal(S)$ such that $g(q)=1$, $|g|<1/2$ on $K$, 
and $g(a_j)=g(b_j)$ for all $j$. Hence there is a unique holomorphic function 
$h\in\Ocal(\Sigma)$ such that $h\circ f=g$. Then $h(p)=1$ and $|h|<1/2$ on $f(K)$, so 
$p\notin \widehat{f(K)}_{\Ocal(\Sigma)}$. If $p=c_j$ for some $j$, 
we choose $g\in\Ocal(S)$ such that $g(a_j)=g(b_j)=1$, $|g|<1/2$ on $K$,
and $g(a_i)=g(b_i)$ for all $i\ne j$; the conclusion is the same as above.
This proves that $f(K)$ is $\Ocal(\Sigma)$-convex, and hence also $\Ocal(X)$-convex.

Conversely, if $K\subset S$ is a compact set such that $f(K)\subset X$ is $\Ocal(X)$-convex,
then $\wt K=f^{-1}(f(K))$ is $\Ocal(S)$-convex. The condition on $f$ implies that $\wt K$ is the union
of $K$ with at most finitely many points $p_1,\ldots, p_m\in S\setminus K$. By the Oka-Weil theorem 
there exists $g\in \Ocal(S)$ such that $g(p_i)$ is close to $1$ for $i=1,\ldots, m$ and $|g|<1/2$ on $K$. 
Hence the points $p_i$ do not belong to the hull of $K$, so $K$ is $\Ocal(S)$-convex. 
\end{proof}

%
%

A compact set $K$ in a topological space $S$ is said to be {\em regular}
if $K$ is the closure of its interior $\mathring K$.

\begin{definition}\label{def:SPP}
A pair $K\subset L$ of compact convex sets in $\R^N$ is a {\em simple convex pair} if 
there are a linear function $\lambda\colon\R^N\to\R$ and constant $a\in\R$ such that
\begin{equation}\label{eq:KL}
	K=\{z\in L : \lambda(z)\le a\}.
\end{equation}
\end{definition}

\begin{lemma}\label{lem:SCPs}
Given regular compact convex sets $C\subset B$  in $\R^N$ and an open set $U\subset \R^N$ containing $C$,
there is a finite sequence of regular compact convex sets $K_1\subset K_{2} \subset\cdots\subset K_{m+1}=B$
such that $C \subset K_1\subset U$ and $(K_{i},K_{i+1})$ is a simple convex pair for every $i=1,\ldots,m$.
\end{lemma}

\begin{proof}
Given a linear function $\lambda\colon\R^N\to\R$ and a number $a\in\R$ we let 
\[
	H(\lambda,a)=\{x\in\R^N:\lambda(x)\le a\}.
\]
Since $C$ is compact and convex, it is the intersection of closed half-spaces.
Hence there exist finitely many linear functions $\lambda_1,\ldots,\lambda_m\colon\R^N\to\R$
and numbers $a_1,\ldots,a_m\in\R$ such that 
\[
	C  \,\,\subset\,\, \bigcap_{i=1}^m H(\lambda_i,a_i) \,\subset\, U. 
\]
The sets $K_i=\bigcap_{j=i}^m H(\lambda_j,a_j) \cap B$ for $i=1,\ldots,m$ and $K_{m+1}=B$ 
then satisfy the lemma. (If $K_i=K_{i+1}$ for some $i$ then $K_{i}$ may be removed from
the sequence.)
\end{proof}

%
%
A compact set $K$ in a complex manifold $S$  is said to be a {\em Stein compact} if $K$ admits a basis of 
open Stein neighborhoods in $S$. If $K\subset A$ are compacts in $S$, we say that $K$ is $\Ocal(A)$-convex if there 
is an open set $W\subset S$ containing $A$ such that $K$ is $\Ocal(W)$-convex. 
By $\Ocal(A)$ we denote the algebra of functions that are holomorphic in open neighborhoods
of $A$ (depending on the function).

The following notion of a {\em special Cartan pair} is a small variation of \cite[Def.\ 5.7.2,\ p.\ 234]{Forstneric2017E}. 
A slightly more restrictive notion was used in \cite{AndristFRW2016}  where the sets
$C\subset B$ were assumed to be smoothly bounded strongly convex,  while $A$ and 
$D=A\cup B$ were strongly pseudoconvex (a {\em strongly pseudoconvex Cartan pair}).
The main novelty is the notion of  a {\em very special Cartan pair} in which the attaching set 
of the convex bump is a thin convex slab; this plays an important role in the proof of
Lemma \ref{lem:main} in the following section.

%
%
\begin{definition}
\label{def:CP}
A pair of compact sets $(A,B)$ in a complex manifold $S$ is a {\em special Cartan pair} if 
\begin{itemize}
\item[\rm(i)]  the sets $A$, $B$, $D=A\cup B$ and $C=A\cap B$ are Stein compacts, 
\vspace{1mm}
\item[\rm(ii)]  $A$ and $B$ are {\em separated} in the sense that
$\overline{A\setminus B}\cap \overline{B\setminus A} =\varnothing$, and 
\vspace{1mm}
\item[\rm (iii)] there is a holomorphic coordinate system on a neighborhood of $B$ in $S$ 
in which $B$ and $C=A\cap B$ are regular convex sets.
\end{itemize}
A special Cartan pair $(A,B)$ with $C=A\cap B$ is {\em very special} if 
\begin{itemize}
\item[\rm (iv)] there is a holomorphic coordinate system on a neighborhood of $B$ in $S$ 
in which $(C,B)$ is simple convex pair (see Definition \ref{def:SPP}).
\end{itemize}
\end{definition}

If $K$ is a compact convex set in $\R^N$, then a {\em slice} of $K$ is the intersection of $K$ with a 
real affine hyperplane, and a {\em slab} of $K$ is a subset of the form
\[
	K_{a,b} = \{x\in K : a\le \lambda(x) \le b\}
\]
where $a<b$ are real numbers and $\lambda\colon \R^N\to \R$ is a  linear function. 
The number $b-a$ is called the {\em thickness} of the slab $K_{a,b}$. If $K$ is a compact subset of a manifold $S$
that is contained in a local chart and is convex in that chart, then a slice or a slab of $K$ 
will be understood as a subset of the respective type in the given chart. 

%
%
\begin{lemma}\label{lem:simple}
Assume that $(A,B)$ is a special Cartan pair in a complex manifold $S$. 
Given an open set $W\subset S$ containing $A$, there is a finite sequence of compact sets 
\[
	A\subset A_1\subset A_2\subset \cdots A_{m+1}= A\cup B
\]
such that $A_1\subset W$, for every $i=1,\ldots,m$ we have
$A_{i+1}=A_i\cup B_i$ where $(A_i,B_i)$ is a very special Cartan pair,
and $C_i=A_i\cap B_i$ is an arbitrarily thin slab of $B_i$. 
If the set $B$ is $\Ocal(D)$-convex where $D=A\cup B$, 
then the pairs $(A_i,B_i)$ can be chosen such that $B_i$ is $\Ocal(A_{i+1})$-convex for
every $i=1,\ldots,m$.
\end{lemma}

\begin{proof}
Let $C=A\cap B$. By the assumption, there are an open neighborhood $V_0 \subset S$ of $B$ and a biholomorphic
map $\theta\colon V_0 \to \wt V_0 \subset \C^d$ onto an open convex subset  of $\C^d$
such that $\theta(C)\subset \theta(B)$ are regular compact convex set in $\C^d$. We use the chart $\theta$ to define
the notion of convexity, slices, slabs, and simple convex pairs in $V_0$. 

Pick an open neighborhood $U$ of $C$, with $U\subset W$, and choose a compact convex set
$\wt C\subset U$ which contains $C$ in its interior.

Lemma \ref{lem:SCPs} furnishes a sequence $K_1\subset K_{2} \subset\cdots\subset K_{m+1}=B$ 
of regular compact convex sets (with respect to the chart $\theta\colon V_0\to \wt V_0$) such that 
\begin{equation}\label{eq:BC} 
	B\cap \wt C \subset K_1\subset U
\end{equation}
and $(K_{i},K_{i+1})$ is a simple convex pair for every $i=1,\ldots,m$
(see Def.\ \ref{def:SPP}). This means that for every $i$ there are an $\R$-linear
function $\lambda_i\colon\C^d\to\R$ and a number $b_i\in\R$ such that
\begin{equation}\label{eq:Ki}
	K_i=\{x\in K_{i+1} : \lambda_i(\theta(x)) \le b_i\}.
\end{equation}
Choose a number $a_i\in \R$ with $a_i<b_i$ and close to $b_i$. For every $i=1,\dots, m$ let
\begin{equation}\label{eq:AiBi}
	A_i=A\cup K_i,\quad  B_i=\{x\in K_{i+1} : a_i \le \lambda_i(\theta(x))\}.
\end{equation}
Assuming that $a_i$ is chosen sufficiently close to $b_i$ for each $i$, 
condition \eqref{eq:BC} implies that 
\[
	A\cap B_i=\varnothing\quad \text{for}\ \ i=1,\ldots,m.
\]
Then $A_i\cup B_i=A\cup K_{i+1}=A_{i+1}\subset A\cup B$ and
\begin{equation}\label{eq:Ci}
	C_i=A_i\cap B_i=\{x\in K_{i+1} : a_i \le \lambda_i(\theta(x)) \le b_i\}.
\end{equation}
Thus, $C_i$ is a slab of the compact convex set $K_{i+1}$.
Note also that $D=A\cup B = A_{m+1}$ which is a Stein compact.
It is easily verified by a downward induction on $i$ that $A_i$ is a Stein compact and 
$(A_i,B_i)$ is a very special Cartan pair for every $i=1,\ldots,m$.
Note that every $B_i$ is a convex subset of $B$ (in the $\theta$-coordinates). 
If $B$ is $\Ocal(D)$ convex, then it clearly follows that $B_i$
is $\Ocal(A_{i+1})$-convex for every $i=1,\ldots,m$.
\end{proof}

\begin{remark}\label{rem:choices}
We have a lot of freedom in the choices of the slabs $C_i$ \eqref{eq:Ci}. In particular, we can replace $a_i$ and $b_i$ 
by any pair of numbers $a'_i,b'_i$ with $a_i<a'_i<b'_i<b_i$ and redefine the sets $K_i$ \eqref{eq:Ki},
$A_i,B_i$ \eqref{eq:AiBi} and $C_i$ \eqref{eq:Ci} accordingly. 
In particular, the attaching slabs $C_i$ of the convex bumps $B_i$ can be chosen 
arbitrarily thin, a fact that will be important in the proof of Lemma \ref{lem:main}
in the following section.
\end{remark}

%
%

\section{The main lemma}
\label{sec:mainlemma}
In this section we prove the following main lemma which implies Theorem \ref{th:main} 
in exactly the same way as \cite[Lemma 2.2]{AndristFRW2016} implies \cite[Theorem 1.1]{AndristFRW2016}.

\begin{lemma}
\label{lem:main}
Assume that $S$ is a complex manifold of dimension $d$, and $X$ is a Stein manifold of dimension $2d$
with the density property or the volume density property. Let $(A,B)$ be a special Cartan pair in $S$
(see Def.\ \ref{def:CP}). Set $C=A\cap B$ and $D=A\cup B$. Assume that 
\begin{itemize}
\item[\rm (a)] $L$ is a compact $\Ocal(X)$-convex set in $X$, 
\item[\rm (b)] $K$ is a compact set contained in $\mathring A\setminus C$ such that $K\cup B$ is $\Ocal(D)$-convex,  
\item[\rm (c)] $W\subset S$ is an open set containing $A$, and 
\item[\rm (d)] $f\colon W\to X$ is a holomorphic map such that $f^{-1}(L)\subset \mathring K$; equivalently,
\begin{equation}\label{eq:collar1}
	f(W\setminus \mathring K) \subset X\setminus L.
\end{equation}
\end{itemize}
Then it is possible to approximate $f$  as closely as desired, uniformly on $A$, by a holomorphic immersion 
$\tilde f\colon \wt W\to X$ on a neighborhood $\wt W$ of $D=A\cup B$ such that  
\begin{equation}\label{eq:collar2}
	\tilde f(\wt W\setminus \mathring K) \subset X\setminus L.
\end{equation}
\end{lemma}

\begin{proof}
By the definition of a very special Cartan pair (see Def.\ \ref{def:CP}) there are an open neighborhood $V_0$ of $B$ 
and a biholomorphic map $\theta\colon V_0 \to \wt V_0 \subset \C^d$ 
onto an open convex subset  of $\C^d$ such that $\theta(C)\subset \theta(B)$ are regular compact convex set in $\C^d$. 
In the sequel, when speaking of convex subsets of $V_0$, we mean sets whose $\theta$-images in $\C^d$ are convex.

Replacing $S$ by a Stein neighborhood of the compact strongly pseudoconvex 
domain $D=A\cup B$, we may assume that $D$ is $\Ocal(S)$-convex. Hence, 
any subset of $D$ which is $\Ocal(D)$-convex is also $\Ocal(S)$-convex.
In particular, this holds for the sets $B$ and $K\cup B$ by the assumption (b).
The same is true for the set $C$ which is convex in $B$.
Furthermore, we claim that $A$ is $\Ocal(S)$-convex. Indeed, given a point $p\in D\setminus A = B\setminus A$,
there is a function $g\in\Ocal(B)$ such that $g(p)=1$ and $|g|<1/2$ on $C$. Since
$B$ is $\Ocal(S)$-convex, we can approximated $g$ uniformly on $B$ by a function
$h\in\Ocal(S)$ satisfying the same conditions. In particular, $|h|<1/2$ on $bA\cap B$
which is the relative boundary of the set $B\setminus A$ in $D$. Since $B\setminus A$ is
a relative neighborhood of $p$ in $D$, Rossi's local maximum modulus principle
implies that $p$ does not belong to the $\Ocal(D)$-convex hull of $A$ and the claim follows.

By Lemma \ref{lem:simple} it suffices to consider the case when $(A,B)$ is a very special Cartan pair.
Indeed, the cited lemma allows us to replace a special Cartan pair by a finite sequence of very
special Cartan pairs, so we obtain a map $\tilde f$ satisfying the conclusion of Lemma \ref{lem:main}
by a finite number of applications of the same lemma for a very special Cartan pair.

Hence we shall assume from now on that $(A,B)$ is a very special Cartan pair.

Let $W\subset S$ be a neighborhood of $A$ as in conditions (c) and (d).
Pick a smoothly bounded strongly pseudoconvex Runge domain $W_0\Subset W$ such that $A\subset W_0$.  
We claim that $f$ can approximated as closely as desired uniformly on $A$ 
by a proper holomorphic immersion $g\colon W_0\hra X$ such that $g^{-1}(L) \subset \mathring K$. 
To see this, pick a strongly plurisubharmonic exhaustion function $\sigma\colon X\to \R$ such that 
$\{\sigma<0\}$ on $L$  and $\sigma >0$ on $f(\overline {W_0 \setminus K})$;
such $\sigma$ exists because $L$ is $\Ocal(X)$-convex and $f(\overline {W_0 \setminus K}) \cap L=\varnothing$.
Given $\epsilon>0$, we can apply \cite[Theorem 1.1]{BDF2010} in order to approximate 
$f$ uniformly on $A$ by a proper holomorphic map $g\colon W_0\to X$ satisfying $\sigma(g(z)) > \sigma(f(z))-\epsilon$ 
for all $z\in W_0$. Choosing $\epsilon>0$ small enough ensures that $g^{-1}(L) \subset \mathring K$;
equivalently, $g(W_0\setminus K)\subset X\setminus L$.
Since $n=2d$, the general position argument shows that $g$ can be chosen an immersion with simple double
points. Replacing $f$ by $g$, we may assume that $f$ satisfies these conditions. 

The image $\Sigma= f(W_0) \subset X$ is a closed immersed complex submanifold of $X$ with 
simple double points. By Lemma \ref{lem:holoconvex} it follows that a compact subset 
$M \subset W_0$ is $\Ocal(W_0)$-convex if and only if its image $f(M)\subset X$ is $\Ocal(X)$-convex. 
By what has been said above, the sets $f(A)$, $f(C)$, and $f(K\cup C)$ are $\Ocal(X)$-convex. 
By  \eqref{eq:collar1})  we also have that $L\cap \Sigma \subset f(K)$, and hence the sets $L'=L\cup f(K)$ and 
$L'\cup f(C)$ are $\Ocal(X)$-convex in view of  \cite[Lemma 6.5]{Forstneric1999}.

At this point we arrive to the main difference with respect to \cite[proof of Lemma 2.2]{AndristFRW2016}.
In that lemma, the map $f\colon W_0\to X$ can be chosen an embedding since $n>2d$. In the present case,
with $n=2d$, it is an immersion with simple double points. However, the attaching set $C=A\cap B$ of the bump $B$ 
can be chosen a thin slab of the convex set $B$ (see Lemma \ref{lem:simple} and Remark \ref{rem:choices}).
A suitable choice of $C$ ensures that $f$ is an embedding on a neighborhood of $C$. 
Indeed, most slices of $B$ (which are convex sets of real dimension $2d-1$) do not contain any of
the finitely many double points of $f$; it then suffices to let $C$ be a sufficiently thin slab 
around such a slice and to adjust the sets $A$ and $B$ accordingly (see Remark \ref{rem:choices}).

We assume from now on that $f(C)$ is embedded in $X$.
Pick a compact set $P\subset X\setminus L'$ containing $f(C)$ 
in its interior such that $L' \cup P$ is also $\Ocal(X)$-convex. 

Choose small open convex neighborhoods $U\subset V$ of the sets $C$ and $B$, respectively, such that  
$U\Subset V\cap W_0$ and $V\Subset V_0$. (The notation $V\Subset V_0$ means that the closure of
$V$ is compact and contained in $V_0$.) We choose $U$ small enough such that $f|_U$ is an embedding
(recall that $f|_C$ is an embedding). The normal bundle of the immersion $f\colon W_0\to X$ 
is holomorphically trivial over the convex set $U$ by the Oka-Grauert principle 
(see \cite[Theorem 5.3.1, p.\ 213]{Forstneric2017E}). Let $\D^{n-d}$ denote the unit polydisc in $\C^{n-d}$. 
It follows that there are a neighborhood $W_1 \subset W_0$ of $A$, a convex neighborhood $U_1\subset U$ of 
$C$, and a holomorphic map $F\colon W_1 \times \D^{n-d} \to X$ such that $F$ is injective 
on $U_1\times \D^{n-d}$ (hence biholomorphic onto the open subset 
$F(U_1\times \D^{n-d})$ of $X$) and $F(z,0)=f(z)$ holds for all $z\in W_1$.  
By a further shrinking of the neighborhood $U_1\supset C$ 
and rescaling of the variable $w\in \D^{n-d}$ we may also assume that the Stein domain 
\begin{equation}\label{eq:Omega}
	\Omega:=F(U_1 \times \D^{n-d})\subset P\subset X\setminus L'
\end{equation}
is Runge in $\mathring P$ and its closure $\overline \Omega$ is $\Ocal(P)$-convex. 
Since $L'\cup P$ is $\Ocal(X)$-convex, it follows that $L' \cup \overline{\Omega}$ is also $\Ocal(X)$-convex. 
Hence there is a Stein neighborhood $\Omega'\subset X$ of $L'$ such that 
$\overline \Omega \cap \overline \Omega'=\varnothing$ and the union 
$\Omega_0:=\Omega\cup\Omega'$ is a Stein Runge domain in $X$. 

Since the sets $U_1\subset V$ are convex, we can find an isotopy $r_t\colon V\to V$ of injective holomorphic 
self-maps, depending smoothly on the parameter $t\in [0,1]$, such that 
\begin{enumerate}
\item $r_0$ is the identity map on $V$, 
\vspace{1mm}
\item $r_t(U_1)\subset U_1$ for all $t\in [0,1]$, and 
\vspace{1mm}
\item $r_1(V)\subset U_1$. 
\end{enumerate}
In fact, in the coordinates on $V_0$ provided by the coordinate map $\theta\colon V_0\to\wt V_0\subset \C^d$ 
we can choose $r_t$ to be a family of linear contractions towards a point in $U_1$.

Consider the  isotopy of biholomorphic maps 
$
	\phi_t\colon V\times \D^{n-d}\to V\times \D^{n-d}
$
defined by 
\begin{equation}
\label{eq:phit}
	\phi_t(z,w)= (r_t(z),w),\quad z\in V,\ w\in \D^{n-d},\ t\in [0,1].  
\end{equation}
Since $r_1(V)\subset U_1$ by the condition (3) above, we have that
\begin{equation}\label{eq:phi1}
	\phi_1(V\times \D^{n-d}) \subset U_1\times \D^{n-d}.
\end{equation}
Recall that $\Omega$ is given by \eqref{eq:Omega} and $\Omega_0=\Omega\cup\Omega'$.
We define a smooth isotopy of injective holomorphic maps $\psi_t\colon \Omega_0 \to X$ ($t\in [0,1]$) 
by
\begin{equation} \label{eq:conj}
		\psi_t = F\circ \phi_t \circ F^{-1}\quad \text{on}\ \ \Omega; \qquad 
		\psi_t=\mathrm{Id}\quad \text{on}\ \ \Omega'.
\end{equation} 
The map $\psi_t$ is defined on $\Omega$ since $r_t(U_1)\subset U_1$ 
(and hence $\phi_t(U_1\times \D^{n-d})\subset U_1\times \D^{n-d}$). 
Note that $\psi_0$ is the identity on $\Omega_0$ and the domain $\psi_t(\Omega_0)$ is Runge in $X$ for all $t\in [0,1]$. 

Assuming that $X$ enjoys the density property, the Anders\'en-Lempert-Forstneri\v c-Rosay theorem 
\cite[Theorem 4.10.5, p.\ 143]{Forstneric2017E} allows us to approximate the map $\psi_1\colon \Omega_0 \to X$
uniformly on compacts  in $\Omega_0$ by holomorphic automorphisms $\Psi\in \Aut(X)$.  
Consider the injective holomorphic map 
\begin{equation}\label{eq:G}
	G = \Psi^{-1}\circ F\circ \phi_1 \colon V\times \D^{n-d}\to X.
\end{equation}
Note that $G$ is well-defined and injective on $V\times \D^{n-d}$ in view of \eqref{eq:phi1}
and because $F$ is injective on $U_1\times \D^{n-d}$. By \eqref{eq:phi1} and \eqref{eq:Omega} we have that
\[
	(F\circ \phi_1)(V\times \D^{n-d}) \cap L' \subset F(U_1\times \D^{n-d})\cap L' = \Omega\cap L'=\varnothing.
\]
Since $\psi_1$ equals the identity map on $\Omega' \supset L'$ by  (\ref{eq:conj}), 
$\Psi$ can be chosen  to approximate the identity as closely as desired 
on a neighborhood of $L'$, so we may assume that 
\[
	G(V\times \D^{n-d}) \subset X\setminus L'. 
\]
From  \eqref{eq:G} and the first equation in (\ref{eq:conj}) we see that
\[
	G=\Psi^{-1}\circ F\circ\phi_1 = \Psi^{-1}\circ \psi_1 \circ F \quad\text{on} \ \ U_1\times \D^{n-d}. 
\]
Since $\Psi^{-1} \circ \psi_1$ is close to the identity map on $\Omega = F(U_1\times \D^{n-d})$, 
$G$ is arbitrarily close to $F$ uniformly on compacts in $U_1\times \D^{n-d}$. 
Hence we can apply \cite[Theorem 4.1]{Forstneric2003Acta} (see also \cite[Theorem 9.7.1, p.\ 432]{Forstneric2017E})
to glue $F$ and $G$ into a holomorphic map $\wt F\colon (A\cup B)  \times \frac{1}{2}\D^{n-d}\to X$ 
such that $\wt F$ is close to $F$ on $A\times \frac{1}{2}\D^{n-d}$ and to $G$ on $B\times \frac{1}{2} \D^{n-d}$. 
The holomorphic map $\tilde f:= \wt F(\cdotp,0) \colon A\cup B =D\to X$ then satisfies the conclusion of
Lemma \ref{lem:main}, except that it need not be an immersion with simple double points; 
this can be achieved by a small perturbation since $2d=n$. 
If the approximations are close enough, then $\tilde f(B)\cap L=\varnothing$ and
$\tilde f^{-1}(L)\subset \mathring K$, so condition \eqref{eq:collar2} holds.

This proves Lemma \ref{lem:main} in the case when $X$ enjoys the density property.
A similar argument applies if $X$ enjoys the volume density property 
with respect to a holomorphic volume form; the details are the same as in  
\cite[proof of Lemma 2.2]{AndristFRW2016}.
\end{proof}

%
%
\begin{proof}[Proof of Theorem \ref{th:main}]
Theorem \ref{th:main} follows from Lemma \ref{lem:main} in exactly the same way as \cite[Theorem 1.1]{AndristFRW2016}
follows from \cite[Lemma 2.2]{AndristFRW2016}; the proof goes as follows. One distinguishes the 
noncritical case and the critical case. In the noncritical case we have a pair of compact strongly pseudoconvex
domains $A\subset A'$ in $S$, given by two sublevel sets of a strongly
plurisubharmonic function without critical points on $\overline{A\setminus A'}$.
We are also given a holomorphic immersion $f\colon A\to X$ on a neighborhood of $A$,
a compact $\Ocal(A)$-convex set $K\subset \mathring A$, and a compact $\Ocal(X)$-convex
set $L\subset X$ such that $f(A\setminus \mathring K)\subset X\setminus L$ (see \eqref{eq:collar1}). 
We wish to approximate $f$ as closely as desired uniformly on $A$ by a holomorphic immersion $\tilde f\colon A'\to X$
satisfying $\tilde f(A'\setminus \mathring K)\subset X\setminus L$ (see \eqref{eq:collar2}). 
As explained in \cite[Proposition 2.3]{AndristFRW2016}, we can obtain $A'$ from $A$ by 
attaching finitely many special convex bumps of the type used in Lemma \ref{lem:main}
so that we have a special Cartan pair at every stage. (Note that
strongly pseudoconvex Cartan pairs, used in \cite{AndristFRW2016}, are also
special Cartan pairs in the sense of Definition \ref{def:CP}.)
Hence, an immersion $\tilde f\colon A'\to X$ with the stated properties is obtained 
by successively applying Lemma \ref{lem:main} finitely many times. The {\em critical case}, which 
amounts to the change of topology at a critical point of a strongly plurisubharmonic exhaustion 
function on $X$, is handled in exactly the same way as in \cite{AndristFRW2016}.
\end{proof}


\subsection*{Acknowledgements}
F.\ Forstneri\v c is partially  supported  by the research program P1-0291 and the research grant 
J1-7256 from ARRS, Republic of Slovenia.


{\bibliographystyle{abbrv} \bibliography{bibliography}}


\vspace*{0.5cm}

\noindent Franc Forstneri\v c

\noindent Faculty of Mathematics and Physics, University of Ljubljana, Jadranska 19, SI--1000 Ljubljana, Slovenia.

\noindent Institute of Mathematics, Physics and Mechanics, Jadranska 19, SI--1000 Ljubljana.

\noindent e-mail: {\tt franc.forstneric@fmf.uni-lj.si}

\end{document}